\documentclass[11pt]{amsart}

\usepackage{amsmath,amssymb,amsthm}
\usepackage[margin=1in]{geometry}
\usepackage{hyperref}
\hypersetup{colorlinks=true,linkcolor=black,citecolor=black,urlcolor=blue}

\newtheorem{theorem}{Theorem}[section]
\newtheorem{lemma}[theorem]{Lemma}

\newtheorem{question}[theorem]{Question}

\theoremstyle{definition}

\newtheorem{remark}[theorem]{Remark}

\DeclareMathOperator{\Aut}{Aut}

\newcommand{\GL}{\mathrm{GL}}
\newcommand{\Fq}{\mathbb{F}_q}

\title[Conjugacy classes of $\GL_n(\mathcal O_2)$ via a canonical section]
{A canonical section method for conjugacy \\ classes of $\GL_n(\mathcal O_2)$}

\author{Pooja Singla}
\address{Department of Mathematics, Indian Institute of Technology Kanpur, Kanpur 208016, India}
\email{psingla@iitk.ac.in}

\date{\today}

\begin{document}
	
	\begin{abstract}
		Let $\mathcal O$ be the ring of integers of a non-Archimedean local field with
		finite residue field, let $\wp$ be its maximal ideal, and let
		$\mathcal O_2 = \mathcal O/\wp^2$. We show that the class equation of
		$\GL_n(\mathcal O_2)$ depends on $\mathcal O$ only through the cardinality
		of its residue field: for two such rings $\mathcal O$ and $\mathcal O'$
		with isomorphic (finite) residue fields, there is a canonical bijection
		between the conjugacy classes of $\GL_n(\mathcal O_2)$ and of
		$\GL_n(\mathcal O_2')$ which preserves the size of every class. The
		argument constructs a section of the reduction map
		$\GL_n(\mathcal O_2)\to \GL_n(\mathcal O_1)$ which is multiplicative on
		the centralizer of any element in its block Jordan canonical form, and
		uses it to transport the classification of conjugacy classes lying above
		a fixed class of $\GL_n(\mathcal O_1)$ from one ring to the other. This
		note records the original argument for this result, developed in the
		author's 2010 doctoral thesis \cite{Thesis}, which predates and is
		independent of two later proofs of closely related statements: the
		Ext-theoretic classification of similarity classes for $n\le 4$ in
		\cite{PSS2015}, and the Hom-theoretic approach of Jambor and Plesken
		\cite{JP2012} for general uniserial rings of length two. We record the
		centralizer-section argument here since it seems to be of independent
		interest and several colleagues have asked to see it in print.
	\end{abstract}
	
	\maketitle
	
	\section{Introduction}
	
	Let $F$ be a non-Archimedean local field with ring of integers $\mathcal O$,
	unique maximal ideal $\wp$, and fixed uniformizer $\pi$. Assume the residue
	field $\mathcal O/\wp$ is finite, of cardinality $q$. For a positive integer
	$\ell$ write $\mathcal O_\ell = \mathcal O/\wp^\ell$, and for a natural
	number $n$ write $\GL_n(\mathcal O_\ell)$ for the group of invertible
	$n\times n$ matrices over $\mathcal O_\ell$. Since $\GL_n(\mathcal O)$ is a
	maximal compact subgroup of $\GL_n(F)$ and every continuous representation
	of $\GL_n(\mathcal O)$ factors through some $\GL_n(\mathcal O)\to
	\GL_n(\mathcal O_\ell)$, the finite groups $\GL_n(\mathcal O_\ell)$ occupy a
	central place in the representation theory of $p$-adic $\GL_n$; see
	\cite{Green55,Hill93,Hill95,Onn08,AOPS} for background and further
	references.
	
	Onn \cite[Conjecture~1.2]{Onn08} conjectured that the isomorphism type of
	the group algebra $\mathbb C[G_\lambda]$ of the automorphism group of a
	finite $\mathcal O$-module of type $\lambda$ depends only on $\lambda$ and
	on $q$, not on $\mathcal O$ itself. In the author's thesis \cite{Thesis},
	this conjecture was verified for $\GL_n(\mathcal O_2)$. It was shown that
	there is a canonical, dimension-preserving bijection between
	$\mathrm{Irr}(\GL_n(\mathcal O_2))$ and $\mathrm{Irr}(\GL_n(\mathcal O'_2))$
	whenever $\mathcal O$ and $\mathcal O'$ have isomorphic residue fields
	\cite[Theorem~1.2.1]{Thesis}, subsequently published as \cite{SinglaJA2010}.
	Since the number of irreducible representations of a finite group equals
	its number of conjugacy classes, this already shows that $\GL_n(\mathcal
	O_2)$ and $\GL_n(\mathcal O'_2)$ have the same number of conjugacy classes.
	Chapter~5 of \cite{Thesis} sharpens this to the level of the full class
	equation:
	
	\begin{theorem}\label{thm:main}
		Let $\mathcal O$ and $\mathcal O'$ be rings of integers of non-Archimedean
		local fields with isomorphic finite residue fields. Then for every $n$
		there is a canonical bijection between the conjugacy classes of
		$\GL_n(\mathcal O_2)$ and those of $\GL_n(\mathcal O'_2)$ which preserves
		the cardinality of each conjugacy class.
	\end{theorem}
	
	A different, independent proof of a closely related statement was later
	given, jointly with Amritanshu Prasad and Steven Spallone, in
	\cite{PSS2015}: there a theory of normal forms for similarity classes in
	$M_n(R)$, $R$ a local principal ideal ring of length two, is developed by
	interpreting similarity classes as extensions of $R[t]$-modules, and used
	to describe the similarity classes in $M_n(R)$ explicitly, together with
	their centralizers, for $n\le 4$ (see also \cite{AOPV09} for the case
	$\ell=2$, $n=3$, over $\mathcal O$ itself, and \cite{Nechaev83} for an
	earlier treatment of $n=3$). That approach enumerates similarity classes
	directly and, in the range $n \le 4$ it covers, gives much more explicit
	information (centralizer orders, transpose-self-duality, polynomiality in
	$q$) than the argument below. It does not, however, produce the bijection
	of Theorem~\ref{thm:main} for general $n$ directly from a structural
	statement about centralizers, which is the content of the present, earlier
	argument. Since this centralizer-section method has not previously
	appeared in print, we record it here.
	
	Independently, Jambor and Plesken \cite{JP2012} gave a third proof of a
	version of this bijection, for general uniserial rings $R$ of length two
	(not necessarily complete discrete valuation rings), by a method
	complementary to both of the above: rather than transporting a canonical
	extension of characters (as here) or classifying normal forms via
	$\mathrm{Ext}$-groups of $k[t]$-modules (as in \cite{PSS2015}), they lift
	a matrix $A$ over the residue field $k$ to a matrix $A_0$ over $R$ chosen
	so that every matrix commuting with $A$ lifts to one commuting with $A_0$,
	and identify the similarity classes of $M_n(R)$ mapping to the class of
	$A$ with the orbits of the conjugation action of $Z_{\GL_n(k)}(A)$ on
	$Z_{M_n(k)}(A)$ — equivalently (as they observe) with the isomorphism
	classes of $n$-dimensional representations of $k[X,Y]$ supported on the
	class of $A$. As \cite[Remark~1.1]{PSS2015} notes, this $\mathrm{Hom}$-based
	approach and the $\mathrm{Ext}$-based approach of \cite{PSS2015} give the
	same count of classes but organise the bijection differently; the
	centralizer-section approach of the present note gives a third,
	representation-theoretic route to the same statement, working directly
	inside $\GL_n(\mathcal O_2)$ rather than through the module category.
	
	The paper is organized as follows. Section~\ref{sec:prelim} recalls the
	facts about centralizers of matrices over $\mathcal O_1$ that are needed.
	Section~\ref{sec:section} constructs the canonical section on which the
	argument turns (Lemma~\ref{lem:section}) and uses it to prove
	Theorem~\ref{thm:main}. Section~\ref{sec:questions} records some questions
	left open by this approach.
	
	\section{Preliminaries on centralizers}\label{sec:prelim}
	
	Throughout, $\kappa:\GL_n(\mathcal O_2)\to\GL_n(\mathcal O_1)$ denotes the
	reduction map, and $\mathcal O_1=\mathcal O/\wp$ is identified with the
	residue field $\Fq$. We write $K=\ker(\kappa)$; the map $A\mapsto I+\pi A$
	gives an isomorphism $M_n(\Fq)\;\widetilde\to\;K$.
	
	Recall that every matrix $X\in M_n(\Fq)$ is conjugate to a block diagonal
	matrix built from blocks
	\[
	J_r(f)=\begin{pmatrix}
		C_f & I & & & \\
		& C_f & I & & \\
		& & \ddots & \ddots & \\
		& & & C_f & I \\
		& & & & C_f
	\end{pmatrix}_{rd\times rd},
	\]
	one for each irreducible factor $f$ of the characteristic polynomial of
	$X$, where $d=\deg f$, $C_f$ is the companion matrix of $f$, and $r$ ranges
	over the parts of a partition attached to $f$; this is the \emph{block
		Jordan canonical form} of $X$, unique up to reordering of blocks
	(\cite[Theorem~2.3.10]{Thesis}; this is the natural extension, via Hensel's
	lemma applied to $f$, of the classical Jordan form for split matrices). We
	call $X$ \emph{split} if its characteristic polynomial splits over $\Fq$,
	in which case each $f$ is linear and $J_r(f)$ reduces to an ordinary
	elementary Jordan block.
	
	Two structural facts about centralizers, both elementary consequences of
	the primary decomposition, will be used repeatedly.
	
	\begin{lemma}[{\cite[Lemma~2.3.2]{Thesis}}]\label{lem:distinct-eigen}
		Let $a_1,\dots,a_l\in\mathcal O_2$ be such that $a_i-a_j$ is invertible in
		$\mathcal O_2$ for $i\neq j$, and let $A=\bigoplus_{i=1}^l A_i$ be a block
		diagonal matrix over $\mathcal O_2$ in which every diagonal entry of $A_i$
		equals $a_i$. Then $Z_{\GL_n(\mathcal O_2)}(A)=\bigoplus_{i=1}^l
		Z_{\GL_{n_i}(\mathcal O_2)}(A_i)$.
	\end{lemma}
	
	\begin{lemma}[{\cite[Lemma~2.3.4]{Thesis}}]\label{lem:toeplitz}
		For a single elementary nilpotent Jordan type, the centralizer of
		$\bigoplus_{i=1}^l N_{n_i}$ in $M_n(R)$, for any commutative ring $R$ with
		unity, consists exactly of the block matrices whose $(i,j)$ block is a
		\emph{rectangular upper Toeplitz matrix} of size $n_i\times n_j$ (i.e.\ a
		matrix constant on diagonals, upper triangular, padded with a zero block if
		$n_i\neq n_j$).
	\end{lemma}
	
	More generally, for $X=\bigoplus_i J_{\lambda_i}(f)$ built from a single
	irreducible $f$, the centralizer $Z_{M_n(\Fq)}(X)$ consists of block upper
	Toeplitz matrices of the corresponding shape with entries in the subring
	$\Fq[C_f]\cong \Fq[t]/f(t)\cong \mathbb F_{q^{\deg f}}$
	(\cite[Theorem~2.3.12]{Thesis}); for $X$ split and already diagonal by
	eigenvalue this is Lemma~\ref{lem:toeplitz} combined with
	Lemma~\ref{lem:distinct-eigen}. The key point for what follows is that
	\emph{the shape of the centralizer of $X$ over $\Fq$ depends only on the
		partition data of $X$ (the degrees of the irreducible factors of its
		characteristic polynomial and the associated partitions), not on any other
		feature of $\mathcal O_1$ or $\Fq$}, and the same block-Toeplitz
	description continues to hold verbatim over $\mathcal O_2$ once $X$ is
	replaced by a suitable lift.
	
	\section{The canonical section and the proof of Theorem~\ref{thm:main}}
	\label{sec:section}
	
	Fix, once and for all, the unique multiplicative section
	$s:\mathcal O_1^*\to \mathcal O_2^*$ of $\kappa$ characterised by
	$s(1)=1$ and $s(xy)=s(x)s(y)$ (see \cite[Prop.~8]{Serre}), extended to a
	map $s:\mathcal O_1\to\mathcal O_2$ by $s(0)=0$. Applying $s$ entrywise
	gives a map $s:M_n(\mathcal O_1)\to M_n(\mathcal O_2)$ which restricts to a
	section of $\kappa$ on $\GL_n(\mathcal O_1)$.
	
	\begin{lemma}\label{lem:section}
		Let $X\in \GL_n(\mathcal O_1)$ be in its block Jordan canonical form. Then
		there is a section $s_X:\GL_n(\mathcal O_1)\to\GL_n(\mathcal O_2)$ of
		$\kappa$ such that
		\[
		s_X(Y)\,s_X(X)=s_X(X)\,s_X(Y)\qquad\text{for every } Y\in
		Z_{\GL_n(\mathcal O_1)}(X),
		\]
		and $s_X$ restricted to $Z_{\GL_n(\mathcal O_1)}(X)$ depends only on $X$
		and on $\mathcal O_1$ (equivalently, only on the isomorphism type of the
		residue field together with the conjugacy class of $X$), not on any further
		choice.
	\end{lemma}
	
	\begin{proof}
		\emph{Case 1: $X$ splits.} Then $X$ is already an ordinary Jordan matrix,
		and we take $s_X=s$, the entrywise multiplicative section fixed above.
		Since $s(0)=0$ and $s(1)=1$, the matrix $s(X)$ is again in Jordan canonical
		form with the same block sizes. If $Y\in Z_{\GL_n(\mathcal O_1)}(X)$ then,
		by Lemmas~\ref{lem:distinct-eigen} and \ref{lem:toeplitz}, $Y$ is a block
		upper Toeplitz matrix (block-diagonal by eigenvalue, Toeplitz within each
		eigenvalue block); applying $s$ entrywise, $s(Y)$ is again block upper
		Toeplitz of the same shape, and $s(X)$ is a diagonal-plus-nilpotent matrix
		of the same shape as $X$. Reapplying Lemmas~\ref{lem:distinct-eigen} and
		\ref{lem:toeplitz}, this time over $\mathcal O_2$, shows $s(X)$ and $s(Y)$
		commute.
		
		\emph{Case 2: $X$ does not split.} We may assume $X=\bigoplus_{i=1}^l
		J_{\lambda_i}(f)$ for a single irreducible polynomial $f$ of degree $d>1$,
		since distinct irreducible factors contribute a direct sum of centralizers
		with no interaction (there is no nonzero $\Fq[x]$-module map between
		$\Fq[x]/f(x)$ and $\Fq[x]/f'(x)$ for $f\neq f'$ irreducible, so the
		argument reduces exactly as in Lemma~\ref{lem:distinct-eigen}). Let
		$\widetilde{\mathcal O}_1$ be the splitting field of $f$ over $\Fq$ and let
		$\widetilde{\mathcal O}_2$ be the corresponding unramified extension of
		$\mathcal O_2$. By Theorem~2.3.11 of \cite{Thesis}, $\widetilde{\mathcal
			O}_1\cong \Fq[C_f]$, and likewise $\widetilde{\mathcal O}_2\cong
		\mathcal O_2[s(C_f)]$. These isomorphisms identify $\GL_t(\widetilde{\mathcal
			O}_1)$ and $\GL_t(\widetilde{\mathcal O}_2)$ (for $t=\sum\lambda_i$) with
		subgroups of $\GL_n(\mathcal O_1)$ and $\GL_n(\mathcal O_2)$ consisting of
		block matrices with entries in $\Fq[C_f]$ and $\mathcal O_2[s(C_f)]$
		respectively, and $X$ becomes a split matrix over $\widetilde{\mathcal
			O}_1$. By Theorem~2.3.12 of \cite{Thesis}, the full centralizer
		$Z_{\GL_n(\mathcal O_1)}(X)$ already lies inside $\GL_t(\widetilde{\mathcal
			O}_1)$. Applying Case~1 over $\widetilde{\mathcal O}_1$ produces a section
		$\tilde s:\GL_t(\widetilde{\mathcal O}_1)\to\GL_t(\widetilde{\mathcal O}_2)$
		commuting with $\tilde s(X)$ on $Z_{\GL_n(\mathcal O_1)}(X)$; any section
		$\GL_n(\mathcal O_1)\to\GL_n(\mathcal O_2)$ of $\kappa$ extending $\tilde s$
		on this subgroup serves as $s_X$.
		
		In both cases the construction of $s_X|_{Z_{\GL_n(\mathcal O_1)}(X)}$ used
		only the multiplicative section of $\mathcal O_1^*$ (equivalently, of
		$\Fq^*$), the block Jordan shape of $X$, and Hensel's lemma applied to $f$
		inside $\Fq[t]$ — none of which depends on any choice beyond the
		isomorphism type of the residue field. Hence this restriction is canonical.
	\end{proof}
	
	We can now prove Theorem~\ref{thm:main}.
	
	\begin{proof}[Proof of Theorem~\ref{thm:main}]
		Fix an isomorphism between the residue fields of $\mathcal O$ and
		$\mathcal O'$; this identifies $\GL_n(\mathcal O_1)$ with
		$\GL_n(\mathcal O'_1)$, and we write $\GL_n(\Fq)$ for either. Let $C$ be
		the set of conjugacy classes of $\GL_n(\Fq)$, represented by matrices $X$
		in block Jordan canonical form. Reduction modulo $\pi$ sends conjugacy
		classes of $\GL_n(\mathcal O_2)$ onto conjugacy classes of $\GL_n(\Fq)$, so
		it suffices to produce, for each $X\in C$, a canonical size-preserving
		bijection $A_X\leftrightarrow A'_X$ between the conjugacy classes of
		$\GL_n(\mathcal O_2)$ and of $\GL_n(\mathcal O'_2)$ lying above the class
		of $X$.
		
		Fix $X\in C$ and let $s_X$, $s'_X$ be the sections of
		Lemma~\ref{lem:section} for $\mathcal O$ and $\mathcal O'$ respectively.
		Every element of $\GL_n(\mathcal O_2)$ reducing to $X$ can be written
		uniquely as $s_X(X)\,i(Y)$ for $Y\in M_n(\Fq)$, where
		$i:M_n(\Fq)\to K\subset \GL_n(\mathcal O_2)$ is the isomorphism $Y\mapsto
		I+\pi Y$; and analogously over $\mathcal O'$. For $g=s_X(A)\,i(B)\in
		\GL_n(\mathcal O_2)$ with $A\in\GL_n(\Fq)$, $B\in M_n(\Fq)$, a direct
		computation gives
		\[
		g\,s_X(X)\,i(Y)\,g^{-1}
		= s_X(A)\,s_X(X)\,s_X(A)^{-1}\; i\bigl(A\cdot(X^{-1}B+Y-B)\bigr).
		\]
		Consequently:
		\begin{itemize}
			\item $g$ conjugates $s_X(X)i(Y_1)$ to $s_X(X)i(Y_2)$ if and only if
			$A\in Z_{\GL_n(\Fq)}(X)$ and $Y_2 = {}^A(X^{-1}B+Y_1-B)$ for some
			$B\in M_n(\Fq)$;
			\item $g$ centralizes $s_X(X)i(Y)$ if and only if $A\in
			Z_{\GL_n(\Fq)}(X)$ and $Y={}^A(X^{-1}B+Y-B)$;
		\end{itemize}
		where we used $A\in Z_{\GL_n(\Fq)}(X)$ together with Lemma~\ref{lem:section}
		to identify $s_X(A)s_X(X)s_X(A)^{-1}=s_X(X)$. Both conditions involve only
		$X$, $Y_1,Y_2,B\in M_n(\Fq)$ and the action of $Z_{\GL_n(\Fq)}(X)$ on
		$M_n(\Fq)$ by $A\cdot Y = {}^AY$ — data entirely internal to $\GL_n(\Fq)$
		and independent of $\mathcal O$ (respectively $\mathcal O'$). The identical
		conditions therefore describe conjugacy in $\GL_n(\mathcal O'_2)$ once
		$s_X$ is replaced by $s'_X$ and $i$ by the corresponding $i':M_n(\Fq)\to
		\GL_n(\mathcal O'_2)$.
		
		Hence the assignment sending the conjugacy class of $s_X(X)i(Y)$ in
		$A_X$ to the conjugacy class of $s'_X(X)i'(Y)$ in $A'_X$ is well defined
		(two elements $s_X(X)i(Y_1)$, $s_X(X)i(Y_2)$ are conjugate in
		$\GL_n(\mathcal O_2)$ exactly when $s'_X(X)i'(Y_1)$, $s'_X(X)i'(Y_2)$ are
		conjugate in $\GL_n(\mathcal O'_2)$), is a bijection $A_X\to A'_X$, and
		sends each class to one of the same size, since the centralizer condition
		above shows the two classes have equal orbit-stabilizer data under
		isomorphic groups $Z_{\GL_n(\Fq)}(X)\ltimes M_n(\Fq)$ acting identically on
		$M_n(\Fq)$. As $X$ ranges over $C$ these bijections combine to the desired
		bijection of Theorem~\ref{thm:main}, and every step in the construction
		depended only on the fixed residue-field isomorphism and on
		Lemma~\ref{lem:section}, hence is canonical.
	\end{proof}
	
	\begin{remark}
		Combined with \cite[Theorem~1.2.1]{Thesis}, Theorem~\ref{thm:main} shows
		that not only the number but the full multiset of conjugacy-class sizes of
		$\GL_n(\mathcal O_2)$ — equivalently, the class equation — depends on
		$\mathcal O$ only through $q=|\mathcal O/\wp|$.
	\end{remark}
	
	\section{Further questions}\label{sec:questions}
	
	The centralizer-section method of Section~\ref{sec:section} suggests
	several questions beyond the scope of this note, already raised in
	\cite[\S 5.2]{Thesis}.
	
	\begin{remark}[Status of Onn's conjecture beyond $\GL_n(\mathcal O_2)$]
		Since \cite{Thesis}, Onn's conjecture itself — that the isomorphism type
		of the group algebra depends on $\mathcal O$ only through $q$ — has
		developed as follows.
		
		\begin{itemize}
			\item[(a)] \textbf{Established beyond $\GL_n(\mathcal O_2)$.} Singla
			\cite{SinglaClassical2012} proved it for classical groups over
			$\mathcal O_2$, and Stasinski and Vera-Gajardo \cite{SVG2019} proved it
			for general reductive group schemes over $\mathcal O_2$.
			
			\item[(b)] \textbf{Established beyond length two.} Onn, Prasad and the
			author \cite{OPS2025} proved the corresponding statement — in the
			stronger, uniform form of an equality of representation \emph{zeta
				functions}, not merely an isomorphism of group algebras — for
			$\GL_3(\mathcal O_\ell)$ and $\mathrm{SL}_3(\mathcal O_\ell)$ (together
			with their unitary analogues $\mathrm{GU}_3$, $\mathrm{SU}_3$) at
			\emph{every} length $\ell$, for residue characteristic $p>3$. This
			settles the $A_2$ case of a conjecture of Avni, Klopsch, Onn and Voll,
			and gives the first verification of Onn's conjecture for $\GL_n$ (with
			$n>2$) at lengths $\ell\ge 3$.
			
			\item[(c)] \textbf{Known to fail: $\mathrm{SL}_3$ at $p=3$.} Ronen
			\cite{Ronen2023} showed that $\mathbb C[\mathrm{SL}_3(\mathbb
			F_3[t]/(t^3))]$ and $\mathbb C[\mathrm{SL}_3(\mathbb Z/27)]$ are not
			isomorphic, consistent with \cite{OPS2025} excluding $p=3$.
			
			\item[(d)] \textbf{Known to fail: $\mathrm{SL}_2$ in even residual
				characteristic.} Hassain and Singla \cite{HS2022} showed that for
			$\mathcal O$ of characteristic zero with even residue cardinality $q$ and
			ramification index $e(\mathcal O)$, the group algebras $\mathbb
			C[\mathrm{SL}_2(\mathcal O/\wp^{2r})]$ and $\mathbb C[\mathrm{SL}_2(\mathbb
			F_q[t]/(t^{2r}))]$ are \emph{not} isomorphic for every $r>e(\mathcal O)$
			— so the equal- and mixed-characteristic representation theories of
			$\mathrm{SL}_2$ genuinely diverge once the length is large enough, even
			though the residue fields agree. This was later sharpened by Hassain
			\cite{Hassain2023}, who gave an explicit construction of the irreducible
			representations of $\mathrm{SL}_2$ over compact discrete valuation rings
			of even residual characteristic and used it to show, concretely, that
			$\mathbb C[\mathrm{SL}_2(\mathbb Z/2^r\mathbb Z)]$ and $\mathbb
			C[\mathrm{SL}_2(\mathbb F_2[t]/(t^r))]$ are already non-isomorphic for
			every $r\ge 4$ (as opposed to the ramification-index-dependent bound of
			\cite{HS2022}).
		\end{itemize}
		
		So some hypothesis beyond ``isomorphic residue field'' is genuinely
		needed once one leaves $\GL_n(\mathcal O_2)$, and small residual
		characteristic remains delicate even for $\GL_n$/$\mathrm{SL}_n$. In
		every setting listed in (a)--(b) above, where the representation-theoretic
		(Onn) statement \emph{does} hold, the corresponding
		\emph{class-equation} statement — whether the full multiset of conjugacy
		class sizes, and not merely their count, depends on $\mathcal O$ only
		through $q$ — remains open; Theorem~\ref{thm:main} of the present note is,
		to our knowledge, still the only case in which it has been verified.
	\end{remark}
	
	\begin{question}
		Does the class equation of $G_{\lambda,F}=\Aut_{\mathcal O}(M_\lambda)$
		depend on $\mathcal O$ only through the cardinality of its residue field,
		for an arbitrary partition $\lambda$ (not just $\lambda=(2,2,\dots,2)$,
		i.e.\ $\GL_n(\mathcal O_2)$)? 
	\end{question}
	More generally, does the class equation of
	a classical or reductive group scheme over $\mathcal O_2$, or of
	$\GL_3(\mathcal O_\ell)$ and $\mathrm{SL}_3(\mathcal O_\ell)$ for general
	$\ell$, depend on $\mathcal O$ only through $q$, in each case where the
	representation-theoretic analogue is now known
	(\cite{SinglaClassical2012,SVG2019,OPS2025})? This is the natural
	analogue, for conjugacy classes, of Onn's Conjecture~1.2 in \cite{Onn08}.
	\begin{question}
		What are the conjugacy classes of $\GL_n(\mathcal O_2)$ for $n\geq 5$,
		explicitly? The present method reduces this to data internal to
		$\GL_n(\Fq)$ (centralizers of elements and their action on $M_n(\Fq)$),
		but does not by itself enumerate the resulting orbits.
	\end{question}
	
	\begin{question}
		Are the problems of determining the conjugacy classes of the following
		families of groups equivalent, in the sense of \cite[Theorem~6.1]{AOPS}
		for representations: (1) $G_{2^n,\Fq((t))}$ for all $n$; (2)
		$G_{k^n,\Fq((t))}$ for all $k,n$; (3) $G_{\lambda,E}$ for all partitions
		$\lambda$ and all unramified extensions $E$ of $\Fq((t))$; (4)
		$G_{2^n,F}$ for an arbitrary non-Archimedean local field $F$?
	\end{question}

	\subsection*{Acknowledgements} This note reproduces, with minor
	reorganisation, the argument of Chapter~5 of the author's doctoral thesis
	\cite{Thesis}, written under the guidance of Amritanshu Prasad at The
	Institute of Mathematical Sciences, Chennai.

\end{document}